\documentclass[review]{elsarticle}
\usepackage{lineno}
\usepackage[latin1]{inputenc}
\usepackage{times}
\usepackage{amssymb,mathrsfs, amsmath}
\usepackage{amsmath,amsthm}
\usepackage{amssymb}
\usepackage{latexsym}
\usepackage{amsfonts}
\usepackage{epsfig}
\usepackage{hyperref}
\usepackage{graphicx}
\usepackage{graphics}
\usepackage[all]{xy}
\usepackage{tikz-cd}
\usepackage[mathscr]{euscript}
\usepackage{mathrsfs}
\newtheorem{thm}{Theorem}[section]

\newtheorem{defn}{Definition}[section]
\newtheorem{prop}{Proposition}[section]

\newtheorem{lem}{Lemma}[section]

\newtheorem{rem}{Remark}[section]

\newtheorem{cor}{Corollary}[section]

\newtheorem{exmpl}{Example}[section]

\journal{... }

\begin{document}
\begin{frontmatter}

\title{ Equivariant  Cohomology of   Lie superalgebras}


%
\author{RB Yadav \fnref{myfootnote}\corref{mycorrespondingauthor}}
\cortext[mycorrespondingauthor]{Corresponding author}
\ead{rbyadav15@gmail.com}
\author{Subir Mukhopadhyay\corref{mycoauthor}}
\ead{smukhopadhyay@cus.ac.in}
\address{Sikkim University, Gangtok, Sikkim, 737102, \textsc{India}}

\begin{abstract}
In this article,  we give Maurer-Cartan characterizations of equivariant Lie superalgebra structures. We introduce equivariant  cohomology  and  equivariant formal deformation theory of  Lie superalgebras.  As an application of  equivariant cohomology we study the equivariant formal deformation theory of Lie superalgebras. As another  application  we characterize equivariant central extensions of Lie superalgebras using second equivariant cohomology. We give some examples of Lie superalgebras with an action of a group and  equivariant formal deformation of a classical Lie superalgebra. 

\end{abstract}

\begin{keyword}
\texttt{Lie superalgebra, cohomology, extension,  formal deformations, Maurer-Cartan equation}
\MSC[2020] 17A70 \sep 17B99 \sep 17B56 \sep 17B70  \sep  17B10 \sep 17B55
\end{keyword}
\end{frontmatter}
\date{}
\section{Introduction}\label{rbsec1}
Graded Lie algebras have been a topic of interest in physics in the context of "supersymmetries" relating particles of differing statistics. In mathematics, graded Lie algebras have been studied in the context of deformation theory, \cite{MR0438925}.

 Lie superalgebras were studied and a classification was given by Kac   \cite{MR486011}. 
   Leits \cite{MR0422372} introduced a cohomology for Lie superalgebras.  Lie superalgebras are also called $ \mathbb{Z}_2$-graded Lie algebras by physicists.

 Algebraic deformation theory was introduced by Gerstenhaber for rings and algebras \cite{MR171807},\cite{MR0207793},\cite{MR240167}, \cite{MR389978}, \cite{MR704600}. Deformation theory of Lie superalgebras was introduced and studied by Binegar \cite{MR871615}.  Maurer-Cartan characterization  was given for  Lie algebra structures by Nijenhuis and Richardson in  \cite{MR0214636} and for associative algebra structures  by Gerstenhaber in  \cite{MR161898}. Such characterization  for Lie superalgebra structures was given   in \cite{MR1028197}. Deformation theory of Lie superalgebras was studied in \cite{MR871615}.\\
Aim  of the present  paper is to give Maurer-Cartan characterization, introduce   equivariant  cohomology, do some  equivariant cohomology computations in lower dimensions,  introduce  equivariant formal deformation theory of Lie superalgebras and give some examples.  Organization of the paper is as follows. In Section \ref{rbsec2}, we recall definition of Lie superalgebra and give some  examples. In Section \ref{MCC}, we give Maurer-Cartan characterization of equivariant Lie superalgebras. In this section we construct a $\mathbb{Z}\times\mathbb{Z}_2$-graded Lie algebra from a $\mathbb{Z}_2$-graded $G$-vector space. We show that class of  Maurer-Cartan elements of this $\mathbb{Z}\times\mathbb{Z}_2$-graded Lie algebra is the class of $G$-equivariant Lie superalgebra structures on $V.$  In Section \ref{rbsec3}, we introduce   equivariant chain complex and   equivariant cohomology of Lie superalgebras. In Section \ref{rbsec4}, we compute cohomology of Lie superalgebras in degree $0$ and dimension $0$, $1$ and $2$.  In Section \ref{rbsec5}, we introduce  equivariant deformation theory of  Lie superalgebras. In this section  we see  that infinitesimals of  equivariant deformations are   equivariant cocycles. In Section \ref{rbsec6}, we study equivalence of two  equivariant formal deformations and prove that infinitesimals of any two equivalent  equivariant deformations are cohomologous.  Also, in this section we give an example of an equivariant formal deformation of a Lie superalgebra.
\section{Lie Superalgebras} \label{rbsec2}
In this section, we recall definitions of Lie superalgebras and  modules over a Lie superalgebras. We recall some examples of Lie superalgebras.  Throughout the paper we denote a fixed field  by $K$.  Also, we denote the ring of formal power series with coefficients in $K$ by $K[[t]]$.
In any $\mathbb{Z}_2$-graded vector space $V$  we use a notation in which we replace degree $deg(a)$ of an element $a\in V$ by $`a'$ whenever $deg(a)$ appears in an exponent; thus, for example $(-1)^{ab}=(-1)^{deg(a)deg(b)}$.
\begin{defn}
  Let $V=V_0\oplus V_1$ and $W=W_0\oplus W_1$ be $\mathbb{Z}_2$-graded vector spaces over a field $K$. A linear map $f:V \to W$ is said to  be homogeneous of degree $\alpha$ if $f(V_\beta)\subset W_{\alpha+\beta}$, for all  $\beta\in \mathbb{Z}_2= \{0,1\}$. We write $(-1)^{\deg(f)}=(-1)^f$. Elements of $ V_\beta$ are called homogeneous of degree $\beta.$
\end{defn}
\begin{defn} A superalgebra is a $\mathbb{Z}_2$-graded vector space $A=A_0\oplus A_1$ together with a bilinear map $m:A\times A\to A$ such that $m(a,b)\in A_{\alpha+\beta},$ for all  $a\in A_{\alpha}$, $b\in A_{\beta}$. 
\end{defn} 
\begin{defn}
A   Lie superalgebra is a superalgebra $L=L_0\oplus L_1$ over a field $K$ equipped with an operation $[-,-]:L\times L\to L$ satisfying the following conditions:
\begin{enumerate}
  \item $[a,b]=-(-1)^{\alpha\beta}[b,a]$, 
  \item $[[a,[b,c]]=[[a,b],c]+(-1)^{\alpha\beta}[b,[a,c]]$, \hspace{3cm}(Jacobi identity)
\end{enumerate}
for all $a\in L_{\alpha}$and $b\in L_{\beta}$.
  Let $L_1$ and $L_2$ be two Lie superalgebras. A homomorphism $f:L_1\to L_2$ is  a $K$-linear map such that $f([a,b])=[f(a),f(b)].$  Given a Lie superalgebra $L$  $[L,L]$ is the vector subspace of $L$ spanned  by the set $\{[x,y]: x,y\in L\}$. A Lie superalgebra $L$ is called abelian if $[L,L]=0.$
\end{defn}
\begin{exmpl}\label{rbLSe1}
 Let $V=V_{\bar{0}}\oplus V_{\bar{1}}$ be a $\mathbb{Z}_2$-graded vector  space, $dim V_{\bar{0}} = m$, $dim V_{\bar{1}} = n$. Consider the  associative algebra $End V$ of all endomorphisms of $V$. 
Define
\begin{equation}
\tag{1}
 End_i \; V = \{a \in End \;V\;|\; a V_s\subseteq V_{i+s}\} \; ,\; i,s\in \mathbb{Z}_2
 \end{equation}
 One can easily verify that  $End\; V=End_{\bar{0} }\; V\oplus End_{\bar{1}} \; V  $.
The bracket $[a,b] = ab - (-1)^{\bar{a} \bar{b}}ba$ makes $End V$ into a Lie superalgebra, denoted by $\ell(V)$ or $\ell(m,n)$. In some (homogeneous) basis of $V$, $\ell(m,n)$ consists of  block matrices of the form $\big(\begin{smallmatrix}\alpha & \beta \\ \gamma & \delta \end{smallmatrix}\big)$, where $\alpha, \beta, \gamma,  \delta$ are matrices of order $m\times m$, $m\times n$, $n\times m$ and $n\times n,$ respectively.
\end{exmpl}
\begin{exmpl}
 Define  a linear function $str : \ell(V)\rightarrow k$,  by 
$str([a,b]) = 0,\; a,b \in \ell(V)$, and $str\; id_V = m - n$.
$str(a)$ is called a supertrace of  $a\in \ell(V) $.
 Consider the subspace
$$s\ell(m,n) = \{a\in \ell(m,n)\;|\; str\; a = 0\}.$$
Clearly,  $s\ell(m,n)$ is an ideal of $\ell (m,n)$ of codimension 1. Therefore   $s\ell(m,n)$ is a subalgebra of $\ell (m,n)$.

For any $\big(\begin{smallmatrix}\alpha & \beta \\ \gamma & \delta \end{smallmatrix}\big)$  in  $\ell(m,n)$  $str \big(\begin{smallmatrix}\alpha & \beta \\ \gamma & \delta \end{smallmatrix}\big) = tr\;\alpha - tr\;\delta$. 
$s\ell(n,n)$ contains the one-dimensional ideal  $\{\lambda I_{2n}: \lambda\in K\}$.
\end{exmpl}

\begin{defn}
  Let $L=L_0\oplus L_1$ be a Lie superalgebra.  Consider the Lie superalgebra structure on $End (V)$ denoted by $\ell(V)$ and  discussed in Example \ref{rbLSe1}.  A $\mathbb{Z}_2$-graded vector space $M=M_0\oplus M_1$ over the field $K$ is called a representation of  (or  module over)  $L$ if there exists a  Lie superalgebra homomorphism $\phi: L\to \ell(V)$. 
  
 It is easy to verify that $M=M_0\oplus M_1$ is module over $L$ if and only if there exists  a linear  map $\Psi:L\otimes M\to M$ such that following condition is  satisfied \\(for simplicity we write $\psi (a\otimes m)=[a,m]$ , $\forall a\in L, m\in M$.)
 $$[a,[b,m]]=[[a,b],m]+(-1)^{ab}[b,[a,m]].$$
  
for all $a\in L_\alpha$, $b\in L_\beta,$  $\alpha,\beta\in\{0, 1\}$, $m\in M$.
\end{defn}
Clearly, every Lie superalgebra is a module over itself. 

\section{$\mathbb{Z}_2$-graded Groups and their Actions on a Lie Superalgebra}
\begin{defn}
We define a $\mathbb{Z}_2$-graded group as a group $G$ having a subgroup $G_{\bar{0}}$ and a subset $G_{\bar{1}}$ such that
\begin{itemize}
\item[a.]  $G=G_{\bar{0}}\cup G_{\bar{1}}$ 
 \item[b.] for all $x\in G_i, y\in G_j$, $xy\in G_{i+j},$ where $i, j, i+j\in \mathbb{Z}_2$.
\end{itemize} 
\end{defn}

\begin{exmpl}
Consider $\mathbb{Z}_6=\{\bar{0}, \bar{1}, \bar{2}, \bar{3}, \bar{4}, \bar{5}\}$. Take $G=\mathbb{Z}_6$, $G_{\bar{0}}=\{\bar{0}, \bar{2}, \bar{4}\}$, $G_{\bar{1}}=\{\bar{1}, \bar{3}, \bar{5}\}$. Clearly,  with this choice of $G_{\bar{0}}$  and $G_{\bar{1}}$, $G$ is a $\mathbb{Z}_2$-graded group.
\end{exmpl}
\begin{exmpl}
Every group $G$ can be seen as  $\mathbb{Z}_2$-graded group with $G_{\bar{0}}=G$  and $G_{\bar{1}}=\emptyset.$
\end{exmpl}
\begin{exmpl}
    Let $V=V_0\oplus V_1$  be $\mathbb{Z}_2$-graded vector space over a field $K$. Denote by $Hom_{\alpha}(V,V)$ the set of all  homogeneous linear transformations from $V$ to $V$ having degree $\alpha.$  Write $G=Hom_{\Bar{0}}(V,V)\cup Hom_{\Bar{1}}(V,V)$. Then $G$ is a $\mathbb{Z}_2$-graded group with $G_{\Bar{0}}=Hom_{\Bar{0}}(V,V)$, and $G_{\Bar{1}}=Hom_{\Bar{1}}(V,V)$.\\
    In particular for any Lie superalgebra $L=L_0\oplus L_1$  the set of all homogeneous Lie superalgebra automorphisms is a $\mathbb{Z}_2$-graded group.
\end{exmpl}
\begin{defn}
A $\mathbb{Z}_2$-graded group $G$ is said to act on a Lie superalgebra $L=L_0\oplus L_1$ if there exits a map $$\psi:G\times L\to L,\;\; (g,x)\mapsto \psi(g,x)=gx$$ satisfying following conditions
\begin{enumerate}
\item $ex=x,$ for all $x\in L.$ Here $e\in G$ is the identity element of $G.$
\item $\forall g\in G_i$, $\psi_g:L\to L$ given by $\psi_g(x)=\psi(g,x)=gx$ is a homogeneous linear map of degree $i$.
\item $\forall g_1, g_2\in G$, $\psi(g_1g_2, x)=\psi(g_1, \psi(g_2, x))$, that is $(g_1g_2)x=g_1(g_2x)$.
\item For $x, y\in L$, $g\in G$, $[gx,gy]=g[x,y].$
\end{enumerate}
We denote an action as above by $(G, L)$.
\end{defn}
\begin{prop}
  Let $G$ be a finite $\mathbb{Z}_2$-graded group and $L$ be a Lie superalgebra. Then $G$ acts on $L$ if and only if there exists a group homomorphism of degree $0$
  $$ \phi: G \to Iso(L,L),\;\; g\mapsto  \phi(g) = \psi_g$$
from the group $G$ to the group of homogeneous Lie superalgebra isomorphisms from $L$ to $L$.
\end{prop}
\begin{proof}
  For an action $(G,L)$, we define a map $\phi:G\to Iso(L,L)$ by $\phi(g)=\psi_g.$ One can verify easily that $\phi$ is a  group homomorphism.
  Now, let $\phi:G\to Iso(L,L)$ be a group homomorphism. Define a map $G\times L\to L$ by $(g,a)\mapsto \phi(g)(a).$ It can be easily seen that this is an action of $G$ on  $L$.
\end{proof}
\noindent\textbf{Note:} In this article we consider action of groups $G$, that is  those  $\mathbb{Z}_2$-graded groups $G$ for which    groups  $G_0=G$ and   $G_1=\emptyset.$
We call a  Lie Superalgebra  $L=L_0\oplus L_1$ with an action of a group $G$ as a    $G$-Lie Superalgebra.
\begin{exmpl}\label{SPLS}
{\bf Super-Poincare algebra}: 
Let  $L_0$ be the vector space  generated by the set $\{J^{\mu\nu}:\mu,\nu=0,1,2,3\}\cup \{P^\mu: \mu=0,1,2,3\}$ over $\mathbb{C}$  and  $L_1$ be the vector pace  generated by the set $\{Q_\alpha: \alpha=1,2\}\cup \{\bar{Q}^{\dot\alpha}:\dot\alpha=1,2\}$ over $\mathbb{C}.$  
The ($ {\mathcal N} = 1$) Super-Poincare algebra $L=L_0\oplus L_1$ is given by\footnote{Here   $\mu, \nu, \rho, \sigma = 0,1,2,3$. $\sigma^i$, $i=1,2,3$ represent Pauli spin matrices and one introduces
$\sigma^\mu = ( \bf{1}, \sigma^i)\quad\text{and}\quad \bar{\sigma}^\mu = ( \bf{1},- \sigma^i) ,\\
\left(\sigma^{\mu\nu}\right)_\alpha^{~\beta} = - \frac{i}{4} (\sigma^\mu\bar{\sigma}^\nu - \sigma^\nu\bar{\sigma}^\mu)_\alpha^{~\beta} ,
\left(\bar{\sigma}^{\mu\nu}\right)_\alpha^{~\beta} = - \frac{i}{4} (\bar{\sigma}^\mu\sigma^\nu - \bar{\sigma}^\nu\sigma^\mu)^{\dot\alpha}_{~\dot\beta}$.
		Spinor indices are denoted by $\alpha, \beta,\dot\alpha, \dot\beta $, they  take values from the set $ \{ 1,2\}$ and  are being raised and lowered by $\epsilon^{\alpha\beta}$ ($\epsilon^{\dot\alpha\dot\beta}$), and $\epsilon_{\alpha\beta}$ ($\epsilon_{\dot\alpha\dot\beta}$). They are antisymmetric and we have chosen
		$\epsilon^{12} = \epsilon^{\dot{1}\dot{2}} = +1$}
		\begin{equation}
		\begin{split}\nonumber
		i[J^{\mu\nu}, J^{\rho\sigma}] &= \eta^{\nu\rho} J^{\mu\sigma} - \eta^{\mu\rho} J^{\nu\sigma} -   \eta^{\sigma\mu} J^{\rho\nu}+  \eta^{\sigma\nu} J^{\rho\mu}, \\
		i[P^\mu, J^{\rho\sigma}] &= \eta^{\mu\rho} P^\sigma - \eta^{\mu\sigma} P^{\rho} ,\quad
		[P^\mu, P^\rho] = 0,\\
		[Q_\alpha, J^{\mu\nu}] &= \sum_{\beta=1,2}\left(\sigma^{\mu\nu}\right)_\alpha^{~~\beta} Q_\beta ,\quad [\bar{Q}^{\dot\alpha}, J^{\mu\nu}] = \sum_{\dot{\beta}=1,2}\left(\bar{\sigma}^{\mu\nu}\right)^{\dot\alpha}_{~~\dot\beta} \bar{Q}^{\dot\beta} \\
		[Q_\alpha, P^\mu] &=0, \quad [\bar{Q}^{\dot\alpha}, P^\mu] = 0 \\
		[Q_\alpha, Q_\beta ] &= 0 ,\quad  [\bar{Q}^{\dot\alpha}, \bar{Q}^{\dot\beta} ] = 0 , 		
		[Q_\alpha, \bar{Q}^{\dot\beta} ] = 2 \sum_{\mu=0,1,2,3}(\sigma^\mu)_{\alpha\dot\beta} P_\mu,
		 \end{split}
		\end{equation}
  where $\eta^{\nu\rho}$, $\eta^{\mu\rho}$, $\eta^{\sigma\mu}$,  $\eta^{\sigma\nu}$, $\eta^{\mu\rho} $, $\eta^{\mu\sigma} $, $\left(\sigma^{\mu\nu}\right)_\alpha^{~~\beta}$,  $\left(\bar{\sigma}^{\mu\nu}\right)^{\dot\alpha}_{~~\dot\beta}$ and   $  (\sigma^\mu)_{\alpha\dot\beta}$ are complex numbers.\\
  Consider the group $\mathbb{Z}_m=\{g^n:  n=0,1,\ldots,m-1\}$, where $g=e^{\frac{2\pi i}{m}}$.
	There exists an action of 	$\mathbb{Z}_m$  on the Super-Poincare algebra $L=L_0\oplus L_1$ given by 
\begin{equation}\nonumber
	(g^n, J^{\mu\nu}) \mapsto J^{\mu\nu}, \quad(g^n, P^\mu) \mapsto P^\mu, \quad 	(g^n,Q_\alpha) \mapsto g^n Q_\alpha, \quad (g^n,\bar{Q}^{\dot{\alpha}})\mapsto  g^{m-n} \bar{Q}^{\dot{\alpha}},
\end{equation}
  for every $n=0,1,\ldots,m-1.$
\end{exmpl}
\begin{exmpl}\label{ELS100}
   Let $e_{ij}$ denote a $2\times 2$ matrix with $(i,j)$th entry $1$ and all other entries $0.$ Consider $L_0=span\{e_{11}, e_{22}\}$, $L_1=span\{e_{12}, e_{21}\}$. Then $L=L_0\oplus L_1$ is a Lie superalgebra with the bracket $[\;,\;]$ defined by $$[a,b]=ab-(-1)^{\bar{a}\bar{b}}ba.$$ Define a function $\psi:\mathbb{Z}_2\times L\to L$ by    
  $ \psi(0,x)=x, \forall x\in L$,  $\psi(1, e_{11})=e_{22}$, $\psi(1, e_{22})=e_{11}$, $\psi(1, e_{12})=e_{21}$, $\psi(1, e_{21})=e_{12}$.
Obviously Conditions $1-3$   hold for $(\mathbb{Z}_2,L)$ to be  an action. To verify condition $4$ it is enough to verify for basis elements of $L_0$ and $L_1$. We have 
\begin{enumerate}
\item $1[e_{ii}, e_{ii}]=0=[1e_{ii}, 1e_{ii}],$ $\forall$  $i=1,2$.
\item $1[e_{ii}, e_{jj}]=0=[e_{jj}, e_{ii}]=[1e_{ii}, 1e_{jj}],$ $\forall$ $i,j=1,2, \; i\ne j$.
\item $1[e_{ij}, e_{ji}]=1(e_{ii}-(-1)^1e_{jj})=e_{jj}+e_{ii}=[1e_{ij}, 1e_{ji}],$  $\forall$ $i,j=1,2, \; i\ne j$.
\item $1[e_{ij}, e_{ij}]=0=[e_{ji}, 1e_{ji}]=[1e_{ij}, 1e_{ij}],$  $\forall$ $i,j=1,2, \; i\ne j$.
\item $1[e_{ii}, e_{ij}]=1(e_{ij})=e_{ji}=[e_{jj}, e_{ji}]=[1e_{ii}, 1e_{ij}],$ $\forall$ $i,j=1,2, \; i\ne j$.
\item $1[e_{jj}, e_{ij}]=1(-e_{ij})=-e_{ji}=[e_{ii}, e_{ji}]=[1e_{jj}, 1e_{ij}],$ $\forall$ $i,j=1,2, \; i\ne j$.
\end{enumerate} 
From above it is clear that  $(\mathbb{Z}_2,L)$ is   an action.
\end{exmpl}
\begin{defn}
Let $L=L_0\oplus L_1$ be a Lie superalgebra. Let $G$ be a finite group which acts on $L$. A $\mathbb{Z}_2$-graded vector space $M=M_0\oplus M_1$ with an action of $G$ is called a $G$-module over $L$ if there exists a  $G$-equivariant  bilinear map $[-,-]:L\times M\to M$ such that following condition is  satisfied
$$[a,[b,m]]=[[a,b],m]+(-1)^{ab}[b,[a,m]],$$
  
for all $a\in L_\alpha$, $b\in L_\beta,$  $\alpha,\beta,\in\{0, 1\}$.
\end{defn}
\begin{exmpl}
   Every $G$-Lie superalgebra is a $G$-module over itself. 
\end{exmpl}
\begin{exmpl}
  Let  $L=L_0\oplus L_1$ be the ($ \mathcal {N} = 1$) Super-Poincare algebra, Example \ref{SPLS}.  Let $M_0$ be the span of  $\{P^\mu: \mu=0,1,2,3\}$ and  $M_1$ be the span of the set  $\{Q_\alpha: \alpha=1,2\}\cup \{\bar{Q}^{\dot\alpha}:\dot\alpha=1,2\}$. Then clearly $M=M_0\oplus M_1$ is a $\mathbb{Z}_m$-module over $L=L_0\oplus L_1$.
\end{exmpl}
\section{Maurer-Cartan Characterization of Equivariant Lie Superalgebra Structures}\label{MCC}
Let $V=V_0\oplus V_1$ and $W=W_0\oplus W_1$ be vector spaces over a field $\mathbb{F}$.   An $n$-linear map $f:V \underset{n\; times}{\underbrace{\times\cdots\times}}V\to W$ is said to  be homogeneous of degree $\alpha$ if $f(x_1,\cdots, x_n)$ is homogeneous in $W$ and  $\deg(f(x_1,\cdots, x_n))-\sum_{i=1}^{n}\deg(x_i)=\alpha \;\; (\text{mod}\;2) $, for homogeneous  $x_i\in V$, $1\le i\le n.$ We denote the  degree of a homogeneous $f$ by $\deg(f)$. For simplicity we write $(-1)^{\deg(f)}=(-1)^f$. 
\begin{defn}
    A finite group $G$ is said to act on a  $\mathbb{Z}_2$-graded vector space  $V=V_0\oplus V_1$ if there exits a map $$\psi:G\times V\to V,\;\; (g,x)\mapsto \psi(g,x)=gx$$ satisfying following conditions
\begin{enumerate}
\item $ex=x,$ for all $x\in V$. Here $e$ is the identity element of $G.$
\item $\forall g\in G$,  $\psi_g: V\to V$ given by $\psi_g(x)=\psi(g,x)=gx$ is a homogeneous linear map of degree $0$.
\item $\forall g_1, g_2\in G$, $\psi(g_1g_2, x)=\psi(g_1, \psi(g_2, x))$, that is $(g_1g_2)x=g_1(g_2x)$.
\end{enumerate}
A  $\mathbb{Z}_2$-graded vector space  $V=V_0\oplus V_1$ with an action of a group $G$ is called a $G$-vector space.
\end{defn}
 Consider the permutation group $S_n.$ For any $X=(X_1,\ldots, X_n)$ with $X_i\in V_{x_i}$ and $\sigma\in S_n$, define 
$$ K(\sigma, X)= card\{ (i,j): i<j,\; X_{\sigma(i)}\in V_1, X_{\sigma(j)}\in V_1,  \; \sigma(j) < \sigma(i)\}, $$
$$\epsilon(\sigma, X)=\epsilon(\sigma)(-1)^{K(\sigma,X)},$$
where $ card\; A $ denotes cardinality of a set $A, $ $\epsilon(\sigma)$ is the signature of $\sigma.$ Also,  define  $\sigma.X=(X_{\sigma^{-1}( 1)},\ldots, X_{\sigma^{-1}( n)}).$
We have following Lemma \cite{MR1028197} 
\begin{lem}\label{SFL1}
    \begin{enumerate}
        \item $K(\sigma \sigma',X)=K(\sigma,X)+K(\sigma',\sigma^{-1}X)\;\;\; (mod\; 2).$ 
        \item $\epsilon(\sigma \sigma', X)=\epsilon(\sigma, X)\epsilon(\sigma', 
\sigma^{-1} X)$.
    \end{enumerate}
\end{lem}
For each  $n\in \mathbb{N},$ define   $\mathcal{F}_{n,\alpha}(V,W)$ as  the vector space of all homogeneous $n$-linear mappings  $f:V \underset{n\; times}{\underbrace{\times\cdots\times}}V\to W$ of degree $\alpha.$ Define  $\mathcal{F}_{n}(V,W)=\mathcal{F}_{n,0}(V,W)\oplus \mathcal{F}_{n, 1}(V,W)$,  $\mathcal{F}_{0}(V,W)=W$ and  $\mathcal{F}_{-n}(V,W)=0$,  $\forall n\in \mathbb{N}.$  Take $\mathcal{F}(V,W)=\bigoplus_{n\in \mathbb{Z}}\mathcal{F}_{n}(V,W)$. 

For $F\in \mathcal{F}_n(V,W)$, $X\in V^n,$ $\sigma\in S_n,$ define 
$$(\sigma.F)(X)=\epsilon(\sigma, X)F(\sigma^{-1}X).$$ 

By using Lemma \ref{SFL1},  one concludes that this defines an action of   $S_n$  on the  $\mathbb{Z}_2$-graded vector space  $\mathcal{F}_n(V,W)$.
Define  $\mathcal{E}_n$ for  $n\in \mathbb{Z}$ as follows:\\
 Set  $\mathcal{E}_n=\{F\in \mathcal{F}_{n+1}(V,V) : \sigma.F=F,\;\forall \;\sigma\in S_{n+1} \}$, for $n\ge 0$ and 
$$\mathcal{E}_n=  \begin{cases}
    V& \textit{if}\; n=-1\\
    0& \textit{if}\;n< -1
\end{cases}.$$
Write $\mathcal{E}=\bigoplus_{\in \mathbb{Z}}\mathcal{E}_n$. Define a product $\circ$ on $\mathcal{E}$ as follows:
For $F\in \mathcal{E}_{n,f}$,  $F'\in \mathcal{E}_{n',f'}$  set 
$$F\circ F'=\sum_{\sigma\in S_{(n,n'+1)}}\sigma.(F*F'), \; \text{where}$$
\begin{small}
    $$F*F'(X_1, \ldots, X_{n+n'+1})=(-1)^{f'(x_1+\cdots+x_n)}F(X_1,\ldots,X_n, F'(X_{n+1},\ldots,X_{n+n'+1})),$$ 
    \end{small} 
    for $X_i\in V_{x_i}$, and $S_{(n,n'+1)}$ consists of permutations $\sigma\in S_{n+n'+1}$ such that  $$\sigma(1)<\cdots<\sigma(n),\;\;\sigma(n+1)<\cdots<\sigma(n+n'+1).$$
Clearly, $F\circ F'\in \mathcal{E}_{(n+n',f+f')}$. We have following Lemma \cite{MR1028197}.
\begin{lem}\label{PLA1}
 For $F\in \mathcal{E}_{n,f}$,  $F'\in \mathcal{E}_{n',f'}$,   $F''\in \mathcal{E}_{n'',f''}$  
 $$(F\circ F')\circ F'' -F\circ (F'\circ F'')= (-1)^{n'n''+f'f''}\{(F\circ F'')\circ F'-F\circ ( F''\circ F')\}.$$
\end{lem}

Using Lemma \ref{PLA1}, we have following theorem   \cite{MR195995}, \cite{MR1028197} 
\begin{thm}\label{MLT3}
$\mathcal{E}$ is a $\mathbb{Z}\times \mathbb{Z}_2$-graded Lie algebra
  with the bracket  $[\;,\;]$ defined  by 
  $$ [F,F']=F\circ F'-(-1)^{nn'+ff'}F'\circ F,$$
  for $F\in \mathcal{E}_{n,f}$,  $F'\in \mathcal{E}_{n',f'}$
\end{thm}
Let $G$ be a finite group acting on the vector spaces $V=V_0\oplus V_1$ and $W=W_0\oplus W_1$.  Denote by $\mathcal{F}_{n}^G(V,W)$ the vector space of $G$-equivariant elements of $\mathcal{F}_{n}(V,W),$ that is $F(gX_1,\ldots, gX_n)=gF(X_1,\ldots, X_n),$ $\forall$ $F\in \mathcal{F}_{n}^G(V,W),$ $(X_1,\ldots, X_n)\in V^n.$ Write $\mathcal{F}^{G}(V,W)= \bigoplus_{n\in \mathbb{Z}}\mathcal{F}_{n}^G(V,W)$. For $\sigma\in S_n$, $g\in G$,  $(X_1,\ldots, X_n)\in V^n$, 
\begin{eqnarray}\label{ML2}
   \sigma.(gX_1,\ldots, gX_n)&=&(gX_ {\sigma^{-1}(1)},\ldots, gX_{\sigma^{-1}(n)})\notag\\
   &=&g(X_ {\sigma^{-1}(1)},\ldots, X_{\sigma^{-1}(n)})\notag\\
   &=&g(\sigma.(X_1,\ldots, X_n)). 
\end{eqnarray}
Let  $F\in \mathcal{E}_{n,f}^G$,  $F'\in \mathcal{E}_{n',f'}^G$.  Clearly,  $F*F'\in \mathcal{E}_{n+n',f+f'}^G$. 
Using Equation \ref{ML2}, we conclude that $F\circ F'\in \mathcal{E}_{n+n',f+f'}^G$. This implies that $[\;,\;]$ defines a product in $\mathcal{E}^G.$ Hence using Theorem \ref{MLT3}, we have following theorem.
\begin{thm}\label{Liesup100}
    $\mathcal{E}^G$ is a $\mathbb{Z}\times \mathbb{Z}_2$-graded  Lie algebra with 
  with the bracket  $[\;,\;]$ defined  by 
  $$ [F,F']=F\circ F'-(-1)^{nn'+ff'}F'\circ F,$$
  for $F\in \mathcal{E}_{n,f}$,  $F'\in \mathcal{E}_{n',f'}$
\end{thm}
Using \cite{MR1028197}, Proposition $(3.1)$, we get following theorem.
\begin{thm}
    Given $F_0\in \mathcal{E}_{(1,0)}^G$, $F_0$ defines on a $\mathbb{Z}_2$-graded $G$-vector space $V$   a $G$-Lie superalgebra structure if and only if  $[F_0,F_0]=0.$
\end{thm}
\begin{rem}
    An element $F_0\in \mathcal{E}_{(1,0)}^G$ which satisfies the equation
    \begin{equation}\label{MCE1}
        [F,F]=0
    \end{equation}
    is  called a Maurer-Cartan element and the Equation \ref{MCE1} is called Maurer-Cartan equation. Thus the class of Maurer-Cartan elements is the class of $G$-Lie superalgebra structures on a $\mathbb{Z}_2$-graded $G$-vector space $V$.
\end{rem}
\section{Equivariant Cohomology of Lie Superalgebras}\label{rbsec3}
  Let $L=L_0\oplus L_1$ be a Lie superalgebra and $M=M_0\oplus M_1$ be a module over $L$. For each $n\ge 0$,  a $K$-vector space  $C^{n}(L;M)$ is defined  as follows:  $C^0(L;M)=M$ and for $n\ge 1$, $C^{n}(L;M)$ consists of  those $n$-linear maps   $ f$ from $L^n$ to $M$  such that
  $$ f(x_1,\ldots,x_i,x_{i+1},\ldots, x_n)=(-1)^{x_ix_{i+1}}f(x_1,\ldots,x_{i+1},x_i\ldots, x_n).$$
  For $n < 0$,  define  $C^{n}(L;M)=0.$ 
Clearly,  $C^{n}(L;M)=C^{n}_0(L;M)\oplus C^{n}_1(L;M)$, where $C^{n}_0(L;M)$ and $ C^{n}_1(L;M)$ are vector subspaces of $C^{n}(L;M)$ containing elements of degree $0$ and $1$, respectively. A  linear map $\delta^n:C^{n}(L;M)\to C^{n+1}(L;M)$ is defined  by (\cite{MR871615}, \cite{MR0422372})
  \begin{eqnarray}\label{ECLS1}
    &&\delta^n f(x_1,\cdots, x_{n+1})\notag\\
     &=& \sum_{i<j}(-1)^{i+j+(x_i+x_j)(x_1+\cdots+x_{i-1})+x_j(x_{i+1}+\cdots+x_{j-1})}\notag\\&&f([x_i,x_j], x_1,\ldots,\hat{x_i},\ldots,\hat{x_j},\ldots,x_{n+1}) \notag\\
     && +\sum_{i=1}^{n+1}(-1)^{i-1+x_i(f+x_1+\cdots+x_{i-1})}[x_i,f(x_1,\ldots,\hat{x_i},\ldots,x_{n+1})],
  \end{eqnarray}  for all homogeneous $f\in C^{n}(L;M)$, $n\ge 1$, and $\delta^0f(x_1)=(-1)^{x_1f}[x_1,f]$, for all $f\in C^0(L;M)=M$. Note that every  $f\in C^{n}(L;M)$  can be written as $f=f_0\oplus f_1$, where $f_0\in C^{n}_0(L;M)$  and $f_1\in C^{n}_1(L;M)$ $\delta^n$ is well defined.   Clearly, for each $f\in C^{n}_G(L;M)$, $n\ge 0,$  $\deg(\delta f)=\deg(f).$

 From \cite{MR871615}, \cite{MR0422372}, we have following theorem: 
\begin{thm}\label{ECLS2}
    $\delta^{n+1}\circ \delta^n=0$, that is, $(C^{\ast}(L;M),\delta)$, where $C^{\ast}(L;M)=\oplus_n C^{n}(L;M)$, $\delta=\oplus_n\delta^n$, is a cochain complex.
  \end{thm}  
Let $G$ be a finite group which acts on $L$. Let $M$ be a $G$-module over $L$.   For each $n\ge 0$, we define  a $K$-vector space  $C^{n}_G(L;M)$ as follows:  $C^0_G(L;M)=M$ and for $n\ge 1$, $C^{n}_G(L;M)$ consists of  those  $ f\in C^n(L,M)$  which are  $G$-equivariant, that is, $f(ga_1, \ldots, ga_n)=gf(a_1, \ldots, a_n)$, for all $(a_1, \ldots, a_n)\in L^{n},$ $g\in G.$
Clearly,  $C^{n}_G(L;M)=(C_G^{n})_0(L;M)\oplus (C_G^{n})_1(L;M)$, where $(C_G^{n})_0(L;M)$ and $ (C_G^{n})_1(L;M)$ are vector subspaces of $C^{n}_G(L;M)$ containing elements of degree $0$ and $1$, respectively. We define a  $K$-linear map $\delta^n_G:C^{n}_G(L;M)\to C^{n+1}_G(L;M)$  by 
$$\delta^n_G f(x_1,\ldots, x_{n+1})=\delta^n f(x_1,\ldots, x_{n+1}) .$$

Clearly, $\delta^n_Gf(gx_1,\ldots, gx_{n+1})=g\delta^n f(x_1,\ldots, x_{n+1})$ for each $f\in C^{n}_G(L;M)$, $g\in G.$ Thus $\delta^n_G$ is well defined.  Write $C^{\ast}_G(L;M)=\oplus_n C^{n}_G(L;M)$, $\delta_G=\oplus_n\delta^n_G$. Using Theorem \ref{ECLS2} we have following theorem:
\begin{thm}
$(C^{\ast}_G(L;M),\delta_G)$ is a cochain complex.
\end{thm}

We denote $\ker(\delta^n_G)$ by $Z^n_G(L;M)$ and image of $(\delta^{n-1}_G)$ by $B^n_G(L;M)$.
We call  the $n$-th  cohomology $Z^n_G(L;M)/B^n_G(L;M)$ of the cochain complex $\{C^{n}_G(L;M),\delta^n_G\}$ as the  $n$-th equivariant cohomology of $L$ with coefficients in $M$ and  denote it by $H^{n}_G(L;M)$. $L$ is a module over itself. So we can consider  cohomology groups $H^{n}_G(L;L)$. We call $H^{n}_G(L;L)$ as the $n$-th equivariant  cohomology group of $L$.
We  have \begin{small}
$$Z^n_G(L;M)=(Z_G^n)_0(L;M)\oplus (Z_G^n)_1(L;M),  \;B^n_G(L;M)=(B_G^n)_0(L;M)\oplus (B_G^n)_1(L;M),$$
\end{small}
where $(Z_G^n)_i(L;M)$and $(B_G^n)_i(L;M)$ are submodules of  $(C_G^n)_i(L;M)$, $i=0,1$. Since boundary map  $\delta^n_G:C^{n}_G(L;M)\to C^{n+1}_G(L;M)$ is homogeneous of degree $0$, we conclude that $H^{n}_G(L;M)$ is $\mathbb{Z}_2$-graded and $$H^{n}_G(L;M)=(H_G^{n})_0(L;M)\oplus (H_G^{n})_1(L;M),$$ where $(H_G^{n})_i(L;M)=(Z_G^n)_i(L;M)/(B_G^n)_i(L;M)$, $i=0,1$.

\begin{rem}
Let   $(L=L_0\oplus L_1, [-,-])$ be a Lie superalgebra with an action of a group $G$. Let $\pi$ be the Maurer-Cartan element $F\in \mathcal{E}_{(1,0)}^G$.       From direct calculation one can verify that if we take representation (or module) $M$ to be $L$, then  $$\delta^n(f)=[F,f],$$
here Lie bracket used in $[F,f]$ is the Lie bracket in Theorem \ref{Liesup100} for the $\mathbb{Z}\times \mathbb{Z}_2$-graded  Lie algebra $\mathcal{E}^G$.
\end{rem}

\section{Equivariant Cohomology of Lie Superalgebras in Low Degrees}\label{rbsec4}
Let $G$ be a finite group and  $L=L_0\oplus L_1$ be a Lie superalgebra with an action of $G$. Let  $M=M_0\oplus M_1$ be a $G$-module over $L.$ For $m\in M_0=(C_G^0)_0(L;M)$, $f\in (C_G^1)_0(L;M)$ and $g\in (C_G^2)_0(L;M)$
\begin{equation}\label{SLCOH1}
  \delta^0_G m(x)=[x,m],
\end{equation}
\begin{equation}\label{SLCOH2}
 \delta^1 f(x_1,x_2)=-f([x_1,x_2])+[x_1,f(x_2)]-(-1)^{x_2x_1}[x_2, f(x_1)],
\end{equation}
\begin{small}
\begin{align} \label{SLCOH3}
 \delta^2 g(x_1,x_2,x_3)=&-g([x_1,x_2],x_3)+(-1)^{x_3x_2}g([x_1,x_3],x_2)-(-1)^{x_1(x_2+x_3)}g([x_2,x_3], x_1)\nonumber\\
   &+[x_1,g(x_2,x_3)]-(-1)^{x_2x_1}[x_2,g(x_1,x_3)]\nonumber\\
   &+(-1)^{x_3x_1+x_3x_2}[x_3,g(x_1,x_2)].
\end{align}
\end{small}
 The set $\{m\in M_0|[x,m]=0, \forall x\in L\}$ is called annihilator of $L$ in $M_0$ and is denoted by $ann_{M_0}L$.  We have \begin{eqnarray*}
          (H_G^0)_0(L;M) &=& \{m\in M_0|[x,m]=0,\;\text{for all}\; x\in L\} \\
          &=& ann_{M_0}L.
         \end{eqnarray*}
 A $G$-equivariant homogeneous linear map $f:L\to M\;$ is called  derivation from $L\;$ to $M\;$ if $f([x_1,x_2])=(-1)^{fx_1}[x_1,f(x_2)]-(-1)^{fx_2+x_2x_1}[x_2, f(x_1)],$ that is $\delta^1_Gf=0.$  For every $m\in M_0$ the map $x\mapsto [x,m]$ is called an inner derivation from  $L$ to $M$. We denote the vector spaces of equivariant derivations and equivariant inner derivations from  $L$ to $M$ by $Der^G(L;M)$ and $Der_{Inn}^G(L;M)$ respectively. By using \ref{SLCOH1}, \ref{SLCOH2} we have $$(H_G^1)_0(L;M)=Der^G(L;M)/Der_{Inn}^G(L;M).$$

Let $L$ be a Lie superalgebra with an action of a finite group $G$ and $M$ be a $G$-module over $L.$ We regard $M$ as an abelian Lie superalgebra with an action of $G$.   An extension of $L$ by $M$ is an exact sequence
  \[\xymatrix{0\ar[r]& M\ar[r]^i &\mathcal{E}\ar[r]^\pi &L\ar[r] &0 }\tag{*}\]
 of Lie superalgebras  such that $$[x,i(m)]=[\pi(x),m].$$
 The exact sequence $(*)$ regarded  as a sequence of $K$-vector spaces, splits. Therefore without any loss of generality we may assume that  $\mathcal{E}$ as a  $K$-vector space coincides with the direct sum $L\oplus M$ and that $i(m)=(0,m),$ $\pi(x,m)=x.$ Thus we have  $\mathcal{E}=\mathcal{E}_0\oplus \mathcal{E}_1,$ where $\mathcal{E}_0=L_0\oplus M_0$, $\mathcal{E}_1=L_1\oplus M_1.$ The multiplication in $\mathcal{E}=L\oplus M$ has then necessarily the form $$ [(0,m_1),(0,m_2)]   =  0,\;  [(x_1,0),(0,m_1)] = (0,[x_1,m_1]),$$
 $$[(0,m_2),(x_2,0)] = -(-1)^{m_2x_2}(0,[x_2,m_2]),\; [(x_1,0),(x_2,0)] = ([x_1,x_2],h(x_1,x_2)),$$  for some $h\in (C_G^2)_0(L;M)$, for all  homogeneous $x_1,x_2\in L$, $m_1,m_2\in M.$
Thus, in general, we have
 \begin{equation}\label{SLCOH7}
[(x,m),(y,n)]=([x,y],[x,n]-(-1)^{my}[y,m]+h(x,y)),
\end{equation}
for all  homogeneous $(x,m)$, $(y,n)$ in $\mathcal{E}=L\oplus M.$\\
 Conversely, let $h:L\times L\to M$ be a bilinear  $G$-equivariant homogeneous map of degree $0$. For homogeneous $(x,m)$, $(y,n)$ in $\mathcal{E}$ we define multiplication in $\mathcal{E}=L\oplus M$ by Equation \ref{SLCOH7}.
 For homogeneous $(x,m)$, $(y,n)$ and $(z,p)$ in $\mathcal{E}$ we have
 \begin{scriptsize}
\begin{eqnarray}\label{SLCOH4}
  &&[[(x,m),(y,n)],(z,p)]\nonumber\\
  &=&([[x,y],z],[[x,y],p]-(-1)^{zx+zn}[z[x,n]]+(-1)^{ym+zy+zm}[z,[y,m]]+[h(x,y),z]+h([x,y],z))\nonumber\\
\end{eqnarray}
\begin{eqnarray}\label{SLCOH5}
&&[(x,m),[(y,n),(z,p)]]\nonumber\\
&=&([x,[y,z]], [x,[y,p]]-(-1)^{nz}[x,[z,n]]-(-1)^{my+mz}[[y,z],m]+[x,h(y,z)]+h(x,[y,z])\nonumber\\
\end{eqnarray}
\begin{eqnarray}\label{SLCOH6}
&&[(y,n),[(x,m),(z,p)]]\nonumber\\
&=&([y,[x,z]],[y,[x,p]]-(-1)^{mz}[y,[z,m]]-(-1)^{nx+nz}[[x,z],n]+[y,h(x,z)]+h(y,[x,z]))\nonumber\\
\end{eqnarray}
 \end{scriptsize}
 From Equations \ref{SLCOH4}, \ref{SLCOH5}, \ref{SLCOH6} we conclude that $\mathcal{E}=L\oplus M$ is a Lie superalgebra  with product given by Equation \ref{SLCOH7}  if and only if $\delta^2_G h=0.$
We denote the Lie superalgebra given by Equation \ref{SLCOH7} using notation $\mathcal{E}_h$. Thus for every cocycle $h\in (C_G^2)_0(L;M)$ there exists an extension
 \label{SLCOH8}
   \[E_h:\xymatrix{0\ar[r]& M\ar[r]^i &\mathcal{E}_h\ar[r]^\pi &L\ar[r] &0 }\]
of $L$ by $M$, where $i$ and $\pi$ are inclusion and projection maps, that is, $i(m)=(0,m),$ $\pi(x,m)=x$.
 We say that two  extensions
  \[\xymatrix{0\ar[r]& M\ar[r] &\mathcal{E}^i\ar[r] &L\ar[r] &0 } \;(i=1,2)\]
  of $L$ by $M$ are equivalent if there is a $G$-equivariant Lie superalgebra isomorphism $\psi:\mathcal{E}^1\to \mathcal{E}^2$ such that following diagram commutes:
\[
\xymatrix{
  0 \ar[r] & M \ar[d]_{Id_M} \ar[r]^-{} & \mathcal{E}^1 \ar[d]_-{\psi} \ar[r]^-{} & L \ar[d]^-{Id_L} \ar[r] & 0 \\
  0 \ar[r] & M \ar[r]_-{} & \mathcal{E}^2 \ar[r]_-{} & L \ar[r] & 0
}
\tag{**}\]
We use $F(L,M)$to denote the set of all equivalence classes of extensions of   $L$ by $M$. Equation \ref{SLCOH7} defines a mapping of $(Z_G^2)_0(L;M)$
 onto $F(L,M)$. If for $h,h'\in (Z_G^2)_0(L;M)$ $E_h$ is equivalent to $E_{h'}$, then commutativity of diagram $(**)$ is equivalent to $$ \psi(x,m)=(x,m+f(x)),$$  for some $f\in (C_G^1)_0(L;M)$.
 We have   \begin{small}
 \begin{eqnarray}
   \psi([(x_1,m_1),(x_2,m_2)]) &=& \psi([x_1,x_2],[x_1,m_2]+[m_1,x_2]+h(x_1,x_2))\nonumber \\
    &=& ([x_1,x_2],[x_1,m_2]+[m_1,x_2]+h(x_1,x_2)+f([x_1,x_2])),\nonumber \\
 \end{eqnarray}
  \begin{eqnarray}
    [\psi(x_1,m_1),\psi(x_2,m_2)] &=& [(x_1,m_1+f(x_1)),(x_2,m_2+f(x_2))] \nonumber\\
     &=& ([x_1,x_2],[x_1,m_2+f(x_2)]+[m_1+f(x_1),x_2]+h'(x_1,x_2)).\nonumber\\
  \end{eqnarray}
      
  \end{small}
 Since $\psi([(x_1,m_1),(x_2,m_2)])=[\psi(x_1,m_1),\psi(x_2,m_2)] $, we have
 \begin{eqnarray}
   h(x_1,x_2)-h'(x_1,x_2) &=& - f([x_1,x_2])+[x_1,f(x_2)]+[f(x_1),x_2]\nonumber\\
   &=&- f([x_1,x_2])+[x_1,f(x_2)]-(-1)^{x_1x_2}[x_2,f(x_1)]\nonumber\\
    &=& \delta^1(f)(x_1,x_2)
 \end{eqnarray}
Thus  two extensions $E_h$ and $E_{h'}$ are equivalent if and only if there exists some $f\in (C_G^1)_0(L;M)$ such that $\delta^1f = h-h'$. We thus have following theorem:

 \begin{thm}
   The set $F(L,M)$ of  all equivalence classes of  extensions of $L$ by $M$ is in one to one correspondence with the cohomology group $(H_G^2)_0(L;M)$. This correspondence $\omega :(H_G^2)_0(L;M)\to F(L,M)$ is obtained  by assigning to each cocycle $h\in (Z_G^2)_0(L;M)$, the extension given by multiplication \ref{SLCOH7}.
\end{thm}

\section{ Equivariant Deformation of  Lie Superalgebras }\label{rbsec5}
Let $L=L_0\oplus L_1$ be a Lie superalgebra. We denote the ring of all formal power series with coefficients in $L$ by $L[[t]]$. Clearly,  $L[[t]]=L_0[[t]]\oplus L_1[[t]]$. So every $a_t\in L[[t]]$ is of the form $a_t=(a_t)_0\oplus (a_t)_1$, where $(a_t)_0\in L_0[[t]]$ and $(a_t)_1\in L_1[[t]]$.

\begin{defn}\label{rb2}
Let $L=L_0\oplus L_1$ be a Lie superalgebra with an action of a finite group $G$.  An equivariant  formal one-parameter deformation of  $L$ is a $K[[t]]$-bilinear map  $$\mu_t : L[[t]]\times L[[t]]\to L[[t]]$$ satisfying the following properties:
\begin{itemize}
  \item[(a)]  $\mu_t(a,b)=\sum_{i=0}^{\infty}\mu_i(a,b) t^i$, for all  $a,b\in L$, where $\mu_i:L\times L\to L$, $i\ge 0$ are  $G$-equivariant  bilinear homogeneous mappings of degree zero  and $\mu_0(a,b)=[a,b]$ is the original  product on L.
\item[(b)] $\mu_t(a,b)=-(-1)^{ab}\mu_t(b,a),$ for all homogeneous  $a,b\in L$.
\item[(c)]
\begin{equation}\label{DLT1}
   \mu_t( a,\mu_t(b,c))=\mu_t(\mu_t(a,b),c)+(-1)^{ab}\mu_t(b,\mu_t(a,c)),
  \end{equation}
for all homogeneous $a,b,c\in L$.
\end{itemize}
The Equation  \ref{DLT1} is equivalent to following equation:
   \begin{eqnarray}\label{rbeqn1}
      &&\sum_{i+j=r} \mu_i( a,\mu_j(b,c))\notag\\
      &=& \sum_{i+j=r}\{\mu_i(\mu_j(a,b),c)-(-1)^{ab}\mu_i(b,\mu_j(a,c))\},
   \end{eqnarray}
for all homogeneous $a,b,c\in L$.

\end{defn}

 Now we define a formal deformation of finite order of a Lie superalgebra $L$.
\begin{defn}\label{rb3}
Let $L$ be a Lie superalgebra with an action of a group $G$.  A  formal one-parameter deformation of order $n$  of  $L$ is a $K[[t]]$-bilinear map  $$\mu_t : L[[t]]\times L[[t]]\to L[[t]]$$ satisfying the following properties:
\begin{itemize}
  \item[(a)]  $\mu_t(a,b)=\sum_{i=0}^{n}\mu_i(a,b) t^i$,  $\forall a,b,c\in L$, where $\mu_i:L\times L\to T$, $0\le i\le n$, are  equivariant $K$-bilinear homogeneous maps of degree $0$, and $\mu_0(a,b)=[a,b]$ is the original product on $L$.
   \item[(b)] $\mu_i(a,b)=-(-1)^{ab}\mu_i(b,a),$ for all homogeneous  $a,b\in L$, $i\ge 0.$
      \item[(c)] \begin{equation}\label{FDLS}
   \mu_t( a,\mu_t(b,c))=\mu_t(\mu_t(a,b),c)+(-1)^{ab}\mu_t(b,\mu_t(a,c)),
  \end{equation}
  for all homogeneous $a,b,c\in L$.
\end{itemize}
\end{defn}

\begin{rem}\label{rbrem1}
  \begin{itemize}
    \item For $r=0$, conditions \ref{rbeqn1} is equivalent to the fact that $L$ is a Lie superalgebra.
    \item For $r=1$, conditions \ref{rbeqn1} is equivalent to
     \begin{eqnarray*}\label{rrbeqn1}
      0&=&-\mu_1(a,[b,c])-[a,\mu_1(b,c)]\\
      &&+ \mu_1([a,b],c)+(-1)^{ab}\mu_1(b,[a,c])+ [\mu_1(a,b),c]+(-1)^{ab}[b,\mu_1(a,c)]\\
       &=&\delta^2\mu_1(a,b,c); \;\text{for all homogeneous}\; a,b,c\in L.
   \end{eqnarray*}
          Thus for $r=1$,  \ref{rbeqn1} is equivalent to saying that $\mu_1\in C_0^2(L;L)$  is a cocycle. In general, for $r\ge 0$, $\mu_r$ is just a 2-cochain, that is,  in $\mu_r\in C_0^2(L;L).$
  \end{itemize}
\end{rem}
\begin{defn}
  The cochain  $\mu_1 \in C_0^2(L;L)$ is called infinitesimal of the   deformation $\mu_t$. In general, if $\mu_i=0,$ for $1\le i\le n-1$, and $\mu_n$ is a nonzero cochain in  $C_0^2(L;L)$, then $\mu_n$ is called n-infinitesimal of the  deformation $\mu_t$.
\end{defn}
\begin{prop}
  The infinitesimal   $\mu_1\in C_0^2(L;L)$ of the  deformation  $\mu_t$ is a cocycle. In general, n-infinitesimal  $\mu_n$ is a cocycle in $C_0^2(L;L).$
\end{prop}
\begin{proof}
  For $n=1$, proof is obvious from the Remark \ref{rbrem1}. For $n>1$, proof is similar.
\end{proof}

By direct computation, for all $1\le m\le 2n$, we get following Equation 
\begin{eqnarray}\label{rbeqn101}
 &&-\mu_1(a,[b,c])-[a,\mu_1(b,c)]+ \mu_1([a,b],c)+(-1)^{ab}\mu_1(b,[a,c])\notag\\&&+ [\mu_1(a,b),c]+(-1)^{ab}[b,\mu_1(a,c)]\notag\\
       &=&   \sum_{\substack{i+j=m\\i,j\ge 1}}\left\{\mu_i( a,\mu_j(b,c))-\mu_i(\mu_j(a,b),c)-(-1)^{ab}\mu_i(b,\mu_j(a,c))\right\}
\end{eqnarray}
 We denote the right hand side  of Equation \ref{rbeqn101} by $ob_L^m$ and call it \textit{Obstruction} of order $m.$ Thus, from Equation \ref{rbeqn101}, we have  
 \begin{equation}\label{rbeqn102}
     \delta^2\mu_1(a,b,c)=ob_L^m(a,b,c)=\frac{1}{2}\sum_{\substack{i+j=m\\i,j\ge 1}}[\mu_i,\mu_j](a,b,c)
 \end{equation}
 In Equation \ref{rbeqn101}, the bracket $[-,-]$ used is the graded Lie bracket  given by Theorem \ref{MLT3}.
\begin{thm}
Obstructions are cocycles, that is $\delta_G^3(ob_L^m)=0$.
\end{thm}
\begin{proof} By using Equation \ref{rbeqn102}, we have 
\begin{align*}
        &\delta_G^3\frac{1}{2}\sum_{\substack{i+j=m\\i,j\ge 1}}[\mu_i,\mu_j]\\
        =&\frac{1}{2}\sum_{\substack{i+j=m\\i,j\ge 1}}[\mu,[\mu_i,\mu_j]]\\
        =&\frac{1}{2} \sum_{\substack{i+j=m\\i,j\ge 1}}\left\{[[\mu,\mu_i],\mu_j]]-[\mu_i,[\mu,\mu_j]]\right\}\\    
        =&\frac{1}{2} \sum_{\substack{i+j=m\\i,j\ge 1}}\left\{[[\mu,\mu_i],\mu_j]]-[[\mu,\mu_j],\mu_i]\right\}\\
        =&0 \qquad \text{(By using graded Jacobi Identity)}
   \end{align*}
\end{proof}
As a consequence of previous theorem, we have following result.
\begin{cor}
 If $H_G^3(L,L)=0,$ then cohomology class of $Ob_L^m$ is $0$, for each $1\le m\le 2n$. Hence, every order $n$ deformation can be extended to an order $n+1$ deformation. 
\end{cor}

\section{Equivalence of Equivariant  Formal  Deformations and Cohomology }\label{rbsec6}
Let   $\mu_t$  and $\tilde{\mu_t}$ be two formal deformations of a Lie superalgebra $L-L_0\oplus L_1$. A formal isomorphism from the deformation $\mu_t$ to $\tilde{\mu_t}$  is a $K[[t]]$-linear automorphism $\Psi_t:L[[t]]\to L[[t]]$ of the  form  $\Psi_t=\sum_{i=0}^{\infty}\psi_it^i$, where each $\psi_i$ is a homogeneous $K$-linear map $L\to L$ of degree $0$, $\psi_0(a)=a$, for all $a\in T$ and $$\tilde{\mu_t}(\Psi_t(a),\Psi_t(b))=\Psi_t\circ\mu_t(a,b),$$ for all $a,b\in L.$
\begin{defn}
  Two  deformations $\mu_t$  and $\tilde{\mu_t}$ of a Lie superalgebra $L$ are said to be equivalent if there exists a formal isomorphism  $\Psi_t$ from $\mu_t$ to  $\tilde{\mu_t}$.
\end{defn}
 Formal isomorphism on the collection of all  formal deformations of a Lie superalgebra $L$ is an equivalence relation.
\begin{defn}
  Any formal deformation of T that is equivalent to the deformation $\mu_0$ is said to be a trivial deformation.
\end{defn}

\begin{thm}
  The cohomology class of the infinitesimal of a  deformation $\mu_t$ of   a Lie Superalgebra $L$ is determined by the equivalence class of $\mu_t$.
\end{thm}
\begin{proof}
  Let  $\Psi_t$ be a formal  isomorphism  from  $\mu_t$ to  $\tilde{\mu_t}$. So  we have, for  all $a,b\in L$,  $\tilde{\mu_t}(\Psi_ta,\Psi_tb)=\Psi_t\circ \mu_t(a,b)$. This implies that \begin{eqnarray*}
     (\mu_1-\tilde{\mu_1})(a,b)&=& [\psi_1a,b]+[a,\psi_1b]-\psi_1([a,b])\\
     &=& \delta^1\psi_1(a,b).
  \end{eqnarray*}
 So we have $\mu_1-\tilde{\mu_1}=\delta^1\psi_1$. This completes the proof.
\end{proof}
Next  we discuss an  example of an equivariant formal deformation of a Lie superalgebra.
\begin{exmpl}
 Let $e_{ij}$ denote a $2\times 2$ matrix with $(i,j)$th entry $1$ and all other entries $0.$ Consider $L_0=span\{e_{11}, e_{22}\}$, $L_1=span\{e_{12}, e_{21}\}$. Then $L=L_0\oplus L_1$ is a Lie superalgebra with the bracket $[,]$ defined by $$[a,b]=ab-(-1)^{\bar{a}\bar{b}}ba.$$ Define a function $\psi:\mathbb{Z}_2\times L\to L$ by    
  $ \psi(0,x)=x, \forall x\in L$,  $\psi(1, e_{11})=e_{22}$, $\psi(1, e_{22})=e_{11}$, $\psi(1, e_{12})=e_{21}$, $\psi(1, e_{21})=e_{12}$.
In Example \ref{ELS100}, we have seen that this gives an action of $\mathbb{Z}_2$ on $L.$ 
Define  a bilinear map $*:L\times L\to L$  by $$e_{ij}*e_{kl}= \begin{cases}
e_{li},~~ \text{if}~ j=k\\
0,~~ \text{otherwise}
\end{cases}.$$
Now define $\mu_1:L\times L\to L$ by $$\mu_1(a,b)=a*b-(-1)^{ab}b*a,$$ for all  homogeneous $a$ , $b$ in $L.$ 
We have 
\begin{enumerate}
\item $1\mu_1(e_{ii}, e_{ii})=0=\mu_1(1e_{ii}, 1e_{ii}),$ $\forall$  $i=1,2$.
\item $1\mu_1(e_{ii}, e_{jj})=0=\mu_1(e_{jj}, e_{ii})=\mu_1(1e_{ii}, 1e_{jj}),$ $\forall$ $i,j=1,2, \; i\ne j$.
\item $1\mu_1(e_{ij}, e_{ji})=1(e_{jj}-(-1)^1e_{ii})=e_{ii}+e_{jj}=\mu_1(1e_{ij}, 1e_{ji}),$  for every $i,j=1,2, \; i\ne j$.
\item $1\mu_1(e_{ij}, e_{ij})=0=(e_{ji}, e_{ji})=\mu_1(1e_{ij}, 1e_{ij}),$  $\forall$ $i,j=1,2, \; i\ne j$.
\item $1\mu_1 (e_{ii}, e_{ij})=1(e_{ji})=e_{ij}=\mu_1(e_{jj}, e_{ji})=\mu_1(1e_{ii}, 1e_{ij}),$ $\forall$ \;$i,\;j=1,2, \\i\ne j$.
\item $1\mu_1(e_{jj}, e_{ij})=1(-e_{ji})=-e_{ij}=\mu_1(e_{ii}, e_{ji}]=\mu_1(1e_{jj}, 1e_{ij}),$ for every $i,j=1,2,\;\; \; i\ne j$.
\end{enumerate} 
Hence $\mu_1$ is $\mathbb{Z}_2$ equivariant.
Define $\mu_t=\mu_0+\mu_1t,$ where $\mu_0(a,b)=[a,b].$ We shall show that $\mu_t$ is an equivariant deformation of $L$ of order $1.$  
To conclude this only thing that we need to show is that \begin{small}
\begin{eqnarray*}\label{rrbeqn101}
     \delta^2\mu_1(a,b,c)&=&-\mu_1(a,[b,c])-[a,\mu_1(b,c)]\\
      &&+ \mu_1([a,b],c)+(-1)^{ab}\mu_1(b,[a,c])+ [\mu_1(a,b),c]+(-1)^{ab}[b,\mu_1(a,c)]\\
       &=&0;~~~~ \;\text{for all homogeneous}\; a,b,c\in L.
   \end{eqnarray*}
We have 
\begin{eqnarray}\label{rrbeqn102}
 \delta^2\mu_1(b,c,a)&=&-\mu_1(b,[c,a])-[b,\mu_1(c,a)]\notag\\
      &&+ \mu_1([b,c],a)+(-1)^{bc}\mu_1(c,[b,a])+ [\mu_1(b,c),a]+(-1)^{bc}[c,\mu_1(b,a)]\notag\\
      &=& (-1)^{ac}\mu_1(b,[a,c])+(-1)^{ac}[b,\mu_1(a,c)]\notag\\
      &&-(-1)^{ab+ac} \mu_1(a,[b,c])+(-1)^{ab+ac}\mu_1([a,b],c)\notag\\ &&-(-1)^{ab+ac} [a,\mu_1(b,c)]+(-1)^{ab+ac}[\mu_1(a,b),c]\notag\\
      &=& (-1)^{ab+ac}\{ -\mu_1(a,[b,c])-[a,\mu_1(b,c)]\notag\\
      &&+ \mu_1([a,b],c)+(-1)^{ab}\mu_1(b,[a,c])+ [\mu_1(a,b),c]+(-1)^{ab}[b,\mu_1(a,c)]\}\notag\\
      &=&(-1)^{ab+ac}\delta^2\mu_1(a,b,c)
\end{eqnarray}
    
\end{small}
\begin{eqnarray}\label{els81}
\delta^2\mu_1(e_{11},e_{12},e_{21})&=& -\mu_1(e_{11},e_{11}+e_{22})-[e_{11},e_{22}+e_{11}]\notag\\
      &&+ \mu_1(e_{12},e_{21})+\mu_1(e_{12},-e_{21})+ [e_{21},e_{21}]+[e_{12},-e_{12}]\notag\\
      &=&0
\end{eqnarray}
\begin{eqnarray}\label{els82}
\delta^2\mu_1(e_{11},e_{21},e_{12})&=& -\mu_1(e_{11},e_{11}+e_{22})-[e_{11},e_{22}+e_{11}]\notag\\
      &&+ \mu_1(-e_{21},e_{12})+\mu_1(e_{21},e_{12})+ [-e_{12},e_{12}]+[e_{21},e_{21}]\notag\\
      &=&0
\end{eqnarray}
\begin{eqnarray}\label{els83}
\delta^2\mu_1(e_{11},e_{12},e_{22})&=& -\mu_1(e_{11},e_{12})-[e_{11},e_{21}]\notag\\
      &&+ \mu_1(e_{12},e_{22})+\mu_1(e_{12},0)+ [e_{21},e_{22}]+[e_{12},0]\notag\\
      &=&0
\end{eqnarray}
\begin{eqnarray}\label{els84}
\delta^2\mu_1(e_{11},e_{22},e_{12})&=& -\mu_1(e_{11},-e_{12})-[e_{11},-e_{21}]\notag\\
      &&+ \mu_1(0,e_{12})+\mu_1(e_{22},-e_{12})+ [0,e_{12}]+[e_{22},e_{21}]\notag\\
      &=&0
\end{eqnarray}
\begin{eqnarray}\label{els85}
\delta^2\mu_1(e_{11},e_{22},e_{21})&=& -\mu_1(e_{11},e_{21})-[e_{11},e_{12}]\notag\\
      &&+ \mu_1(0,e_{21})+\mu_1(e_{22},-e_{21})+ [0,e_{21}]+[e_{22},-e_{12}]\notag\\
      &=&0
\end{eqnarray}
\begin{eqnarray}\label{els86}
\delta^2\mu_1(e_{11},e_{21},e_{22})&=& -\mu_1(e_{11},-e_{21})-[e_{11},-e_{12}]\notag\\
      &&+ \mu_1(-e_{21},e_{22})+\mu_1(e_{21},0)+ [-e_{12},e_{22}]+[e_{21},0]\notag\\
      &=&0
\end{eqnarray}
\begin{eqnarray}\label{els87}
\delta^2\mu_1(e_{22},e_{12},e_{21})&=& -\mu_1(e_{22},e_{11}+e_{22})-[e_{22},e_{22}+e_{11}]\notag\\
      &&+ \mu_1(-e_{12},e_{21})+\mu_1(e_{12},e_{21})+ -[e_{21},e_{21}]+[e_{12},e_{12}]\notag\\
      &=&0
\end{eqnarray}
\begin{eqnarray}\label{els88}
\delta^2\mu_1(e_{22},e_{21},e_{12})&=& -\mu_1(e_{22},e_{11}+e_{22})-[e_{22},e_{22}+e_{11}]\notag\\
      &&+ \mu_1(e_{21},e_{12})+\mu_1(e_{21},-e_{12})+ [e_{12},e_{12}]+[e_{21},-e_{21}]\notag\\
      &=&0
\end{eqnarray}
Using Equations \ref{rrbeqn102}, \ref{els81}, \ref{els82}, \ref{els83}, \ref{els84}, \ref{els85}, \ref{els86}, \ref{els87} and  \ref{els88} we conclude that Equation \ref{rrbeqn101} holds. Hence $\mu_t$ is an equivariant deformation  of $L$ of order  $1$.
\end{exmpl}

\section{Declarations}
\noindent\textbf{Funding and/or Conflicts of interests/Competing interests:}\\
The authors declare that they have no known competing financial interests or personal relationships that could have appeared to influence the work reported in this paper.

\noindent \textbf{ Data Availability:}\\
Data sharing not applicable to this article as no datasets were generated or analyzed during the study in this article.

\bibliography{raj.bib}

\end{document}